\documentclass[a4paper,12pt]{article}

\newtheorem{theorem}{Theorem}[section]

\newtheorem{corollary}[theorem]{Corollary} 
\newenvironment{proof}[1][Proof]{\textbf{#1.} }{\  
 \rule{0.5em}{0.5em}\\}

\newcommand{\R}{\mathbb{R}}

\newcommand{\intbar}{-\kern-1em\int}

\usepackage{amsfonts}
\usepackage{xcolor}
\usepackage{hyperref}
\usepackage{enumerate}

\title{A note on the existence of global solutions for reaction-diffusion equations with almost-monotonic
  nonlinearities}

 \author{An\'{\i}bal Rodr\'{\i}guez-Bernal \footnote{Dpto.\ de
    Matem\'atica Aplicada. Univ.\ Complutense de Madrid and Instituto de Ciencias
    Matem\'aticas CSIC-UAM-UC3M-UCM. Spain. arober@mat.ucm.es} \and Alejandro
  Vidal-L\'opez \footnote{Mathematics Insitute.  Univ.\ of
    Warwick. UK. A.Vidal-Lopez@warwick.ac.uk}}

\newcommand{\eps}{\varepsilon}
\renewcommand{\epsilon}{\varepsilon}

\begin{document}

\maketitle

\section{Introduction}

Let us consider the following problem
  \begin{equation}
  \label{eq:pb}
  \left\{\begin{array}{rclcl}
    u_t  - \Delta u  & =
    & f(x,u) & \mathrm{in} & \Omega\\
u&=& 0 &
    \mathrm{on} & \Gamma\\
    u(0) &=& u_0 & \mathrm{in} & \Omega
  \end{array}\right.
\end{equation}
in a bounded domain $\Omega \subset \R^N$, $N\geq 1$, with $f$
 of the form
  \begin{equation} \label{eq:decomposition_f}
    f(x,u) = g(x) + m(x) u + f_0(x,u), \quad x\in \Omega, \ u \in \R, 
  \end{equation}
  where 
  \begin{equation}  \label{eq:assumptions_g_m}
  g \in L^{q_{0}}(\Omega), \   \mbox{with
    $N/2 < q_{0} < \infty$}, \quad m\in L^{r_{0}}(\Omega)
  \ \mbox{with $N/2< r_{0} \leq \infty$}
\end{equation}
and $f_0$ is a Carath\'eodory function, locally Lipschitz in the second
variable in a $L^{q_0}(\Omega)$ form with respect to $x\in\Omega$ , i.e, such
that for $|u|\leq R$ and $|v|\leq R$ we have
\begin{displaymath}
|f_{0}(x,u)-f_{0}(x,v)| \leq L_{0}(x,R) |u-v|
\end{displaymath}
with $0\leq L_{0} (\cdot, R)   \in L^{q_{0}} (\Omega)$ for each $R>0$ ,
and with $f_0(\cdot,0) = \partial_u f_0(\cdot,0) = 0$. 

Our goal here is to prove global existence and uniqueness of solutions for initial
data $u_{0}\in L^{q}(\Omega)$ for any $1\leq q < \infty$ under the
only additional assumption  that $f_{0}(x,u)$ is  \emph{almost monotonic}, i.e, to
satisfy the following condition
\begin{equation}
  \label{eq:f:almostMonotonic}
  \partial_u f_{0} (x,u) \leq L(x), \quad \mathrm{for\ all}\quad x\in
  \Omega,  \ u\in \R
\end{equation}
for some $L\in L^{\sigma_{0}}(\Omega)$, $\sigma_{0}>N/2$. Note that this implies 
\begin{equation}
  \label{eq:f:signCond}
  uf(x,u)\leq C(x) u^2 +D(x)|u|, \quad x\in \Omega, \ u \in \R, 
\end{equation}
for $C= m + L \in L^{\sigma}(\Omega)$, $\sigma=\min\{r_{0},
\sigma_{0}\} > N/2$ and   $0\leq D=|g| \in L^{q_{0}}(\Omega)$.

 \section{Preliminary results} 
 \label{sec:known-results-subcr}
\setcounter{equation}{0}

Notice that, the Nemytskii operator associated to $F(x,u)=g(x) +
f_0(x,u)$ maps bounded sets of $L^\infty(\Omega)$ into bounded sets of
$L^{q_0}(\Omega)$ and so we have the existence of local solutions of
problem (\ref{eq:pb}) for smooth initial data. Namely, we have (see
\cite{henry81:_geomet})

\begin{theorem}
  \label{th:localExistence}
Assume that $f$ is as in
(\ref{eq:decomposition_f})--(\ref{eq:assumptions_g_m}).   
If $\alpha <1$ is such that $2 \alpha - \frac{N}{q_{0}}>0$ , then for any
initial data 
$u_0 \in H^{2\alpha,q_{0}}(\Omega) \cap H^{1,q_{0}}_0(\Omega)$ there exists a unique local 
solution of (\ref{eq:pb}) with initial data $u_0$, with  $u(\cdot;u_0)\in
C([0,\tau);H^{2\alpha,q_{0}}(\Omega) \cap H^{1,q_{0}}_0(\Omega)) \cap
C((0,\tau);H^{2,q_{0}}(\Omega))$ and $u_{t} \in
C((0,\tau);H^{2\gamma,q_{0}}(\Omega) \cap H^{1,q_{0}}_0(\Omega))$ for
any $\gamma <1$, for some $\tau$ depending on  $u_{0}$. 

The solution  satisfies the Variation
of Constant Formula,
\begin{displaymath}
u(t;u_0) = S_{m}(t) u_0 + \int_0^t S_{m}(t-s) \big( g + f_{0}(u(s;u_0)) \big) \,\mathrm{d}s
\end{displaymath}
where $S_{m}$ denotes the semigroup generated by $\Delta+m(x) I$ with Dirichlet boundary
conditions.  
\end{theorem}

Note that since $q_{0} >N/2$, $H^{2\alpha,q_{0}}(\Omega) \cap
H^{1,q_{0}}_0(\Omega) \subset C_{0}^{\eta}(\overline{\Omega})$ for any $0 \leq
\eta < 2-N/q_{0}$. In particular the solution satisfies $u\in C([0,\infty)
\times \overline{\Omega})$. Also, if $g$ and $L_0(\cdot,R)$, for each $R$, are a
bounded functions, then the above holds for any $q_{0} > N/2$.

It is known that condition (\ref{eq:f:signCond}) ensures the global existence of
the local solutions in Theorem \ref{th:localExistence} (see \cite{Arrieta:2004p863} and
\cite{Rodr'iguez-Bernal2008}). Namely, we have
\begin{theorem}
  \label{th:globalExist}
  Assume that $f$ is as in
  (\ref{eq:decomposition_f})--(\ref{eq:assumptions_g_m}) and satisfies
  (\ref{eq:f:signCond}), that is  
\begin{displaymath}
  uf(x,u)\leq C(x) u^2 +D(x)|u|, \quad x\in \Omega, \ u \in \R, 
\end{displaymath}
for some $C\in  L^{\sigma}(\Omega)$ and   $0\leq D \in
L^{\rho}(\Omega)$ with  $\sigma, \rho > N/2$.  

Then for the solutions in Theorem \ref{th:localExistence} one has
$\tau= \infty$.   
\end{theorem}

\section{Existence in $L^q(\Omega)$, $1\leq q<\infty$}
\label{sec:local-exist-lr-supercrit}
\setcounter{equation}{0}

The goal of the section is to prove main result in this work. Namely, the
existence of a unique solution of problem (\ref{eq:pb}) starting at $u_0\in
L^q(\Omega)$ for any $1\leq q<\infty$.

\begin{theorem} 
\label{thr:truncated_problem_well_posed}
Let $1 \leq q < \infty$.  Suppose that $f$ is as in
(\ref{eq:decomposition_f})--(\ref{eq:assumptions_g_m}) and $f_{0}$ satisfies
(\ref{eq:f:almostMonotonic}) with $L\in L^{\sigma_{0}}(\Omega)$,
$\sigma_{0} > N/2$. Then, for
any $u_{0} \in L^{q}(\Omega)$, there exists a solution of (\ref{eq:pb}) defined
for all $t\geq 0$, $u$, such that
\begin{displaymath}
  u  \in C([0,\infty);L^q(\Omega)) \cap
  C((0,\infty);H^{2,q_{0}}(\Omega)\cap H^{1,q_{0}}_0(\Omega)),\quad   u(0) =
  u_{0}  
\end{displaymath}  
and satisfies  
\begin{equation} 
  \label{eq:VCF:th:exists}
    u(t) = S_{m}(t)u_0+ \int_0^t S_{m}(t-s) \big( g + f_{0}(\cdot,u(s)) \big)\,\mathrm{d}s  
\end{equation}
where $S_{m}$ denotes the semigroup generated by $\Delta+m(x) I$ with Dirichlet boundary
conditions. 

Moreover, for each $T>0$ there exists $c(T)$ such that 
\begin{equation}
  \label{eq:Linfty_bounds}
  |u(t,x)| \leq c(T) \left(1+    t^{-\frac{N}{2q}}
    \|u_{0}\|_{L^q(\Omega)} \right) ,
  \qquad 0 <t \leq T  \quad  \textrm{for all  } x \in
  \overline{\Omega} . 
\end{equation}
\end{theorem}
\begin{proof}
  We proceed in several steps. In the first step, fixed $1\leq q<\infty$, we
  construct a Cauchy sequence of approximating solutions.
  Then, we obtain a uniform
  $L^\infty(\Omega)$ bound for the approximating sequence. In a third step, we
  show that the limit of the approximating solutions is a solution of the limit
  problem (notice that such limit exists since the approximating solutions forms
  a Cauchy sequence). Finally, we show how to obtain more regularity of the
  solution constructed in the previous steps.

Without loss of generality we can assume that $L$
  in (\ref{eq:f:almostMonotonic}) is non-negative.
  
\medskip 
\noindent {\bf Step 1. Approximate the initial data.} 
Let $\alpha <1$ such that $2\alpha - \frac{N}{q_{0}}>0$. Then, by Theorem
\ref{th:localExistence}, the problem (\ref{eq:pb}) is well-posed in
$H^{2\alpha,q_{0}}(\Omega) \cap H^{1,q_{0}}_0(\Omega)$. Also, since
$f$ satisfies (\ref{eq:f:signCond}), the 
solutions are globally defined  for $t>0$, see Theorem
\ref{th:globalExist} and (\ref{eq:f:signCond}). 

Hence, for any $1\leq q < \infty$ and $ u_0 \in L^q(\Omega)$, we can
take smooth enough initial data $u_0^n \in 
H^{2\alpha,q_{0}}(\Omega) \cap H^{1,q_{0}}_0(\Omega)$ such that
$u_0^n \to u_0$ in 
$L^q(\Omega)$ as $n\to\infty$ and consider the solutions of (\ref{eq:pb})
starting at $u^n_0$. We define $u_n(t) = u(t;u^n_0)$.

Let $v_{n,k}(t) = u_n(t) - u_k(t)$. Subtracting equations
for $u_n$ and $u_k$, we have 
\begin{displaymath}
  \left\{\begin{array}{rclcl}
  \partial_t v_{n,k}(t) -\Delta v_{n,k}(t)&=& 
m(x) v_{n,k}(t) + f_{0}(x,u_n(t)) -
  f_{0}(x,u_k(t)) & \mathrm{in} & \Omega \\
v_{n,k} &= & 0 &     \mathrm{on}  & \Gamma\\
  v_{n,k}(0) &=& u^0_n - u^0_k & \mathrm{in} & \Omega . 
\end{array}\right.
\end{displaymath}

Observe that for fixed $n,k$,   we have for almost all $x\in \Omega,
t\in [0,T]$ and  some $0<\theta(t,x)<1$ 
\begin{eqnarray*}
  f_{0}(x,u_n(t,x)) - f_{0}(x,u_k(t,x)) &=& 
\partial_u f_{0}(x,\theta(t,x) u_n(t,x) +(1-\theta(t,x)) u_k(t,x))
v_{n,k}(t,x) \\
&=&C_{n,k}(t,x) v_{n,k}(t,x)
\end{eqnarray*}
for some function $C_{n,k}(t,x)$. Notice that $C_{n,k}\in
L^\infty((0,T);L^{q_0}(\Omega))$ since $u_n$ and $u_k$ are smooth,
$f_{0}$ is locally Lipschitz in the second variable in an
$L^{q_0}(\Omega)$ manner. Also, from (\ref{eq:f:almostMonotonic}) we
have $C_{n,k}(x,t) \leq L(x)$ for all $t\geq 0$ and $x\in \Omega$.

Now, consider the linear  problem
\begin{equation}
\label{eq:auxProblem:z_C}
   \left\{\begin{array}{rclcl}
   z_t -\Delta z&=& \big( m(x) + C_{n,k}(t,x)\big) z & \mathrm{in} & \Omega\\
z &= & 0 &  \mathrm{on}  & \Gamma\\
  z(0) &=& z_0  & \mathrm{in} & \Omega
\end{array}\right.
\end{equation}
with $z_{0}$ smooth and denote by $z(t,0;z_0)$ the  solution  whose existence follows from
\cite{Robinson2007} or \cite{brezis96}). Such solutions satisfy by
comparison $z(t,0;-|z_{0}|) \leq  z(t,0;z_{0}) \leq z(t,0;|z_{0}|)$, 
i.e, 
\begin{displaymath}
  |z(t,0;z_{0})| \leq z(t,0; |z_{0}|)
\end{displaymath}
and the latter is a nonnegative solution of (\ref{eq:auxProblem:z_C}).

But for nonnegative initial data, $z_{0} \geq 0$, since $C_{n,k}(x,t) \leq 
L(x)$ for all $t\geq 0$ and $x\in \Omega$, we can
compare $z(t,0;z_{0}) \geq 0$ with the solutions of 
\begin{displaymath}
   \left\{\begin{array}{rclcl}
   w_t -\Delta w&=& \big( m(x) + L(x) \big) w & \mathrm{in} & \Omega\\
w &= & 0  & \mathrm{on}  & \Gamma\\
  w(0) &=& z_0  & \mathrm{in} & \Omega
\end{array}\right.
\end{displaymath}
to obtain $0\leq z(t,0;z_0) \leq w(t;z_0)$. 

Hence, we obtain that for any smooth initial data $z_{0}$ in
(\ref{eq:auxProblem:z_C}) we have 
 \begin{displaymath}
  |z(t, 0;z_0)|\leq w(t;|z_0|) \quad \mbox{for $t\geq 0$}. 
\end{displaymath}
In particular, 
\begin{displaymath}
\|z(t, 0; z_{0})\|_{L^q(\Omega)} \leq
\|w(t;|z_{0}|)\|_{L^q(\Omega)}\leq c \mathrm{e}^{-\lambda 
  t}\|z_0\|_{L^q(\Omega)}   
\end{displaymath}
where $\lambda$ is the first eigenvalue of
$-\Delta - (m(x)+ L(x))I$ on $\Omega$ with Dirichlet boundary conditions.

Now, $v_{n,k}$ is a solution of (\ref{eq:auxProblem:z_C}) and so
\begin{displaymath}
  \|v_{n,k}(t)\|_{L^q(\Omega)}
\leq c\mathrm{e}^{-\lambda t}\|v_{n,k}(0)\|_{L^q(\Omega)}
\end{displaymath}
for all $t\geq 0$. In particular, given $T>0$, we have that for any
$0\leq t\leq T$, 
\begin{displaymath}
  \|v_{n,k}(t)\|_{L^q(\Omega)} 
\leq  c(T)
  \|v_{n,k}(0)\|_{L^q(\Omega)} \to 0 \quad \textrm{as} \quad n,k\to\infty
\end{displaymath}
and so, $u_n$ is a Cauchy sequence in $C([0,T];L^q(\Omega))$. 

Hence,
there exists $u\in C([0,\infty);L^q(\Omega))$ such that for any $T>0$,
\begin{equation} \label{eq:convergence_2_vK_LrOmega} 
  \sup_{t\in [0,T]} \|u_n(t) - u(t)\|_{L^q(\Omega)}  \leq
  c(T)\|u_n^0 - u_0\|_{L^q(\Omega)}\to 0 \quad
  \textrm{as} \quad n\to\infty
\end{equation}
i.e, for any $T>0$, 
\begin{displaymath}
  u_{n} \to u \quad \mbox{in $C([0,T]; L^{q}(\Omega))$} . 
\end{displaymath}
In particular, passing to a subsequence if needed, $u_n(t,x) \to
u(t,x)$ as $n\to\infty$ a.e. for $(t,x)\in[0,T]\times \Omega$.

Also it is easy to see that $u$ does not depend on the sequence of
initial data, but only on $u_{0} \in L^{q}(\Omega)$.

\noindent {\bf Step 2. $L^\infty$-bound for the approximating sequence.} 
Let us  show now that the sequence $u_n(t)$ is uniformly bounded in 
$L^\infty(\Omega)$ with respect to $n$, for $0<\eps\leq t\leq T$. 

For this, since $f$ satisfies (\ref{eq:f:signCond}),  we will use the auxiliary problem
\begin{equation}
  \label{eq:pb:lin:bound}
   \left\{\begin{array}{rclcl}
    U_t  - \Delta U  & =
    & C(x)   U + D(x) & \mathrm{in} & \Omega\\
    U&=& 0 &
    \mathrm{on} & \Gamma\\
    U(0) & &  \textrm{given in   }L^{q}(\Omega) 
  \end{array}\right.
\end{equation}
with  $C= m + L \in L^{\sigma}(\Omega)$, $\sigma=\min\{r_{0},
\sigma_{0}\} > N/2$ and   $0\leq D=|g| \in L^{q_{0}}(\Omega)$, $q_{0}
> N/2$. 

 Denote by $U^{n}(t,x)$ the solution
of (\ref{eq:pb:lin:bound}) with initial data $|u_{0}^{n}|$ and by $U(t,x)$ the solution
of (\ref{eq:pb:lin:bound}) with initial data $|u_{0}|$. 

Now, using the variation of constants formula in
(\ref{eq:pb:lin:bound}) we have 
\begin{displaymath}
  U^n(t) = \Phi(t) + U^n_h(t), \quad  U(t) = \Phi(t) + U_h(t)
\end{displaymath}
where $U^n_h(t), U_h(t)$ are  the solutions of the homogeneous problem 
\begin{displaymath}
   \left\{\begin{array}{rclcl}
    V_t  - \Delta V  & =
    & C(x)V  & \mathrm{in} & \Omega\\
    V&=& 0 &
    \mathrm{on} & \Gamma\\
    V(0) & &  \textrm{given in   }L^{q}(\Omega) 
  \end{array}\right.
\end{displaymath}
resulting from
taking $D\equiv 0$ in (\ref{eq:pb:lin:bound}) and initial data
$|u_{0}^{n}|$ and $|u_{0}|$ respectively,  and $\Phi(t)$ is the unique solution
of problem (\ref{eq:pb:lin:bound}) with $U(0)=0$ (which does not depend on
$u_{0}^{n}$ or $u_{0}$), that is,
\begin{displaymath}
   \left\{\begin{array}{rclcl}
    W_t  - \Delta W  & =
    & C(x)W+ D(x) & \mathrm{in} & \Omega\\
    W&=& 0 &
    \mathrm{on} & \Gamma\\
    W(0) &= & 0 & \mathrm{in} & \Omega. 
  \end{array}\right.
\end{displaymath}
In other words $U^n_h(t) = S_{C}(t)|u_{0}^{n}| $, $U_h(t) =
S_{C}(t)|u_{0}|$ and $\Phi(t)= \int_{0}^{t} S_{C}(t-s) D \, \mathrm{d}s$ where $S_{C}$ denotes the semigroup generated by $\Delta+C(x) I$ with Dirichlet boundary
conditions.  Hence  standard estimates  implies that, for any $T>0$, 
\begin{displaymath}
  \|U^n(t)\|_{L^\infty(\Omega)} \leq c(T) \big( 1  +  
  t^{-\frac{N}{2q}}   \|u_{0}^{n}\|_{L^q(\Omega)} \big), \qquad 0 <t
  \leq T   , 
\end{displaymath}
\begin{displaymath}
  \|U^n(t)\|_{L^q(\Omega)} \leq c(T) \big( 1  +
  \|u_{0}^{n}\|_{L^q(\Omega)} \big), \qquad 0 \leq t
  \leq T   
\end{displaymath}
and 
\begin{eqnarray*}
  \|U^n(t) - U(t)\|_{L^\infty(\Omega)}
&=& \|U^n_h(t) - U_h(t)\| _{L^\infty(\Omega)}
= \|S_{C}(t)(|u_0^n| - |u_0|)\|_{L^\infty(\Omega)}\\
&\leq& c(T) t^{-\frac{N}{2q}} \||u_{0}^{n}| - |u_0|\|_{L^q(\Omega)} \to 0
\quad  \textrm{as} \quad n\to\infty 
\end{eqnarray*}
for $0 <t  \leq T$, and for $0 \leq t  \leq T$
\begin{displaymath}
  \|U^n(t) - U(t)\|_{L^q(\Omega)}
\leq c(T) \||u_{0}^{n}| - |u_0|\|_{L^q(\Omega)} \to 0 \quad
  \textrm{as} \quad n\to\infty.
\end{displaymath}
Therefore, for any $0<\eps < T <\infty$, 
\begin{displaymath}
  U^{n} \to U \quad \textrm{in} \ L^{\infty}([\epsilon,T] \times \Omega) \cap C([0,T]; L^{q}(\Omega))  . 
\end{displaymath}

Observe now that, since $f$ satisfies (\ref{eq:f:signCond}),
$U^{n}(t,x)$ is a supersolution of problem (\ref{eq:pb}) and 
$-U^n(t,x)$ is a subsolution.
Thus,
\begin{equation} \label{eq:bound_Linfty_u_n}
  |u_n(t,x)| \leq  U^{n}(t,x) \leq c(T)\big( 1 +    t^{-\frac{N}{2q}}
   \|u_{0}^{n}\|_{L^q(\Omega)}  \big), \qquad 0 <t \leq T,
\quad \textrm{a.e.\ in  }\Omega, 
\end{equation}
and so
\begin{displaymath}
\|u_n(t)\|_{L^\infty(\Omega)}\leq
c(\epsilon,T,\|u_{0}^{n}\|_{L^q(\Omega)}), 
\quad \epsilon\leq t \leq T.
\end{displaymath}

Now, since $u_0^n \to u_0$ in $L^q(\Omega)$ as $n\to\infty$ and the
convergences $U^{n}(t,x) \to U(t,x)$ and  $u_n(t,x) \to  v(t,x)$ obtained above (see
(\ref{eq:convergence_2_vK_LrOmega}))  we get  
\begin{equation}
  \label{eq:bound:vK:Linfty}
  |u(t,x)| \leq U(t,x) \leq c(T) \big(1 +    t^{-\frac{N}{2q}}
  \|u_{0}\|_{L^q(\Omega)}  \big)  ,
  \qquad 0 <t \leq T,  \quad  \textrm{for a.e.   } x \in\Omega.
\end{equation}

Now observe that the bounds above, the regularity of $u_{n}$ in
Theorem \ref{th:localExistence}  and  (\ref{eq:convergence_2_vK_LrOmega})
imply that for any $0< \epsilon < T < \infty$ and $1\leq s < \infty$,  
\begin{equation} \label{eq:convergence_2_vK_LqOmega} 
  \sup_{t\in [\epsilon,T]} \|u_n(t) - u(t)\|_{L^s(\Omega)}  
\to 0 \quad
  \textrm{as} \quad n\to\infty
\end{equation}
i.e, for any $T>0$ and $1\leq s < \infty$,
\begin{displaymath}
  u_{n} \to u \quad \mbox{in $C([\eps,T]; L^{s}(\Omega))$} . 
\end{displaymath}
In particular $u\in C((0,\infty);L^s(\Omega))$ for any $1\leq s <
\infty$.

\medskip 
\noindent {\bf Step 3. The limit is a solution of (\ref{eq:pb}).}
First, assume $0<\eps<t<T$. Then for any $\phi\in H^{2,q_0^\prime}(\Omega)\cap
H^{1,q_0^\prime}_0(\Omega)$, where $q_0^\prime$ is the conjugate of $q_0$, i.e,
$\frac{1}{q_0}+ \frac{1}{q_0^\prime}=1$ (as usual, for $q_0=1$ we take $q_0^\prime
=\infty$), we have from (\ref{eq:pb}) and the regularity if $u_{n}$ in Theorem
\ref{th:localExistence},
\begin{displaymath}
   \frac{\mathrm{d}}{\mathrm{d}t}  \int_\Omega u_n \phi +
   \int_\Omega u_n (-\Delta \phi)  
=\int_\Omega f(\cdot,u_n)\phi  = \int_\Omega g \phi  +
\int_\Omega m(x) u_n\phi + \int_\Omega f_{0}(x,u_n)\phi .  
\end{displaymath}

Now, using the uniform bounds in (\ref{eq:bound_Linfty_u_n}),
(\ref{eq:bound:vK:Linfty}) and the convergence in
(\ref{eq:convergence_2_vK_LqOmega}), and the fact that $f_{0}$ is locally
Lipchitz in its second variable in an $L^{q_0}(\Omega)$ manner, we have that for
$1\leq s \leq q_0$,
\begin{displaymath}
   f_{0}(\cdot,u_n) \to  f_{0}(\cdot,u) \quad \textrm{in} \quad C([\epsilon,T];
   L^{s}(\Omega)) .
\end{displaymath}

Hence,  letting  $n\to\infty$,we get 
\begin{displaymath}
  \frac{\mathrm{d}}{\mathrm{d}t}  \int_\Omega u \phi +
   \int_\Omega u (-\Delta \phi) 
=\int_\Omega f(\cdot,u)\phi = \int_\Omega g \phi  +
\int_\Omega m(x) u\phi + \int_\Omega f_{0}(x,u)\phi .   
\end{displaymath}

 Notice that  from \cite{Ball1977}, this implies 
\begin{equation}
  \label{eq:VCF:eps}
  u(t) = S_{m}(t-\epsilon)u(\epsilon) + \int_\epsilon^t S_{m}(t-s)h(s)\,\textrm{d}s
\end{equation}
where $S_{m}(t)$ denotes the strongly continuous analytic semigroup generated by
$\Delta + m(x)I$  with homogeneous Dirichlet boundary conditions, and
$h(\cdot) = g + f_{0}(\cdot, u(\cdot))\in
L^{\infty}([\epsilon,T];L^{q_{0}}(\Omega))$.

The smoothing effect of the semigroup gives that 
\begin{displaymath}
\int_\epsilon^t 
S_{m}(t-s)h(s)\,\textrm{d}s \in C([\eps,T];H^{2\gamma,q_{0}}(\Omega)
\cap H^{1,q_{0}}_0(\Omega)), \ \mbox{for any $\gamma <1$}, 
\end{displaymath}
 while  the continuity of the linear semigroup $S_{m}(t)$ at $0$ and
$u(\epsilon)\to u_{0}$ in $L^{q}(\Omega)$ as $\epsilon\to 0$, give, 
taking $\epsilon \to 0$ in (\ref{eq:VCF:eps}),
\begin{displaymath}
  \int_0^t S_{m}(t)h(s)\,\mathrm{d}s 
= \lim_{\epsilon\to 0} \int_\epsilon^t S_{m}(t-s)h(s)\,\textrm{d}s = u(t) -S_{m}(t)u_0.
\end{displaymath} 
Thus, 
\begin{displaymath}
  u(t) = S_{m}(t)u_0 + \int_0^t S_{m}(t-s) \big(  g + f_{0}(s,
  u(s)) \big) \,\textrm{d}s . 
\end{displaymath}

\medskip 
\noindent {\bf Step 4.  Further regularity.} 
From the smoothing effect of the semigroup $S_{m}(t)$ and the
regularity observed above, we have that for any $\eps >0$, $u(\eps) \in
 H^{2\alpha,q_{0}}(\Omega) \cap H^{1,q_{0}}_0(\Omega)$ for some
 $\alpha <1$ such that $2 \alpha - \frac{N}{q_{0}}>0$. 

Therefore, for $t\geq \eps$, $u(t)$ coincides with the unique solution
in Theorems  \ref{th:localExistence} and \ref{th:globalExist}. In
particular $u(t)$ is continuous in $\Omega$ and we can take $x\in
\overline{\Omega}$ in (\ref{eq:bound:vK:Linfty}). 
\end{proof}

\begin{corollary}
  \label{corol:smoothing_estimates}
  For $1\leq q < \infty$ and  $T>0$, we have that the   solution $u$
  in Theorem \ref{thr:truncated_problem_well_posed} satisfies,  for $q\leq
  p\leq\infty$, 
  \begin{displaymath}
    \|u(t)\|_{L^p(\Omega)} 
    \leq c(T) \left(1+
      t^{-\frac{N}{2}\left(\frac{1}{q}-\frac{1}{p}\right)} 
    \|u_0\|_{L^q(\Omega)}\right) \quad 0<t\leq T.
  \end{displaymath}
\end{corollary}
\begin{proof}
  Following the argument in Step 2 in the proof of Theorem
  \ref{thr:truncated_problem_well_posed}, we can bound the
  approximating sequence $u_n$ using the bound
  provided by the linear problem (\ref{eq:pb:lin:bound}) to get
  \begin{displaymath}
    \|u_n(t)\|_{L^q(\Omega)}\leq \|U^n(t)\|_{L^q(\Omega)}
\leq c(T) \big( 1 +  \|u_0\|_{L^q(\Omega)}\big) , \quad 0\leq t\leq T,
\end{displaymath}
  for all $n\geq 0$. Therefore, since $u\in C([0,T];L^q(\Omega))$ is
  the limit of $u_n$ in $C([0,T];L^q(\Omega))$ as $n\to \infty$, we have
  \begin{displaymath}
    \|u(t)\|_{L^q(\Omega)}\leq c(T) \big( 1 +
    \|u_0\|_{L^q(\Omega)}\big) , \quad 0\leq t\leq T. 
  \end{displaymath}
From (\ref{eq:bound:vK:Linfty}) we also have that
  \begin{displaymath}
   \|u(t)\|_{L^\infty(\Omega)}\leq c(T)
    \left(1+t^{-\frac{N}{2q}} \|u_0\|_{L^q(\Omega)} \right), \quad 0< t\leq T. 
  \end{displaymath}
Thus, by interpolation
\begin{displaymath}
  \|u(t)\|_{L^p(\Omega)}\leq \|u(t)\|^{\frac{q}{p}}_{L^q(\Omega)} 
\|u(t)\|^{1-\frac{q}{p}}_{L^\infty(\Omega)} 
\leq c(T)
\left(1+t^{-\frac{N}{2}\left(\frac{1}{q}-\frac{1}{p}\right)}
\|u_0\|_{L^q(\Omega)} \right), \quad 0< t\leq T.  
\end{displaymath}
\end{proof}

Let show now that the solutions of (\ref{eq:pb}) in $L^q(\Omega)$ are
unique for $1\leq q<\infty$. 

\begin{theorem}\label{thr:uniqueness}
  Given $u_{0} \in L^{q}(\Omega)$, $1\leq q < \infty$, there exists a unique
  function
\begin{displaymath}
  v \in
  C([0,\infty);L^q(\Omega))\cap L^\infty_{\mathrm{loc}}((0,\infty);L^\infty(\Omega)), \quad    v(0) = u_{0} 
\end{displaymath}  
that satisfies  
\begin{equation} \label{eq:VCF_truncated_problem}
    v(t) = S_{m}(t)u_0+ \int_0^t S_{m}(t-s) \big( g +
    f_{0}(\cdot,v(s)) \big)    \,\textrm{d}s, \qquad t\geq 0  
\end{equation}
where $S_{m}$ denotes the semigroup generated by $\Delta+m(x) I$ with Dirichlet boundary
conditions.

Therefore the function $u(\cdot)$ constructed in Theorem
\ref{thr:truncated_problem_well_posed} is the unique 
function satisfying this. 
\end{theorem}
\begin{proof} 
  Notice that the function $u$ constructed in Theorem
  \ref{thr:truncated_problem_well_posed} satisfies the assumptions
  above. So, let $v$ also satisfy the statement of
  the theorem. Then from (\ref{eq:VCF_truncated_problem}) we have
  that, for any $\eps>0$, 
  \begin{displaymath}
      v(t) = S_{m}(t-\epsilon)v(\epsilon) + \int_\epsilon^t S_{m}(t-s)
      \big( g +
    f_{0}(\cdot,v(s)) \big)\,\textrm{d}s . 
  \end{displaymath}

From the assumptions on $v$ we have that $h(s) = g +
f_{0}(\cdot,v(s))$ satisfies, for any $T>0$, that $h\in
L^{\infty}([\epsilon,T];L^{q_{0}}(\Omega))$. Then, the  smoothing
effect of the semigroup gives that  
\begin{displaymath}
v \in C([\eps,T];H^{2\gamma,q_{0}}(\Omega)
\cap H^{1,q_{0}}_0(\Omega)), \ \mbox{for any $\gamma <1$}. 
\end{displaymath}
Hence, for $t\geq \eps$,  $v$ is a solution as in Theorem
\ref{th:localExistence}.

Hence, arguing as in (\ref{eq:convergence_2_vK_LqOmega}) we have 
\begin{displaymath}
\sup_{\epsilon\leq t\leq T}\| u(t) - v(t)\|_{L^q(\Omega)} 
\leq c(T) \|u(\epsilon) - v(\epsilon)\|_{L^q(\Omega)} 
\end{displaymath}
with $c(T)$ not depending on $\epsilon$.

The continuity of $u$ and $v$ at $0$ in $L^{q}(\Omega)$, and the fact
that $u(0)=v(0)$ imply $u=v$.
\end{proof}

\section{Final remarks and examples}

(i) Note that Theorems \ref{thr:truncated_problem_well_posed} and
\ref{thr:uniqueness} allow to define a strongly continuos nonlinear semigroup  in
$L^{q}(\Omega)$ as 
\begin{displaymath}
  S(t) u_{0} = u(t;u_{0}), \qquad t\geq 0
\end{displaymath}
where $u(t;u_{0})$ is the solution in Theorem
\ref{thr:truncated_problem_well_posed}. 

The asymptotic behavior of this semigroup is the same as the semigroup obtained
for more regular initial data from Theorems \ref{th:localExistence} and
\ref{th:globalExist}. In fact, from (\ref{eq:Linfty_bounds}) we get that for any
$0< \eps < T < \infty$ and for any bounded set of initial data $B \subset
L^{q}(\Omega)$ we get that
\begin{displaymath}
  \{S(t)B, \ \eps\leq t \leq T\} \quad \mbox{is bounded in
    $L^{\infty}(\Omega)$} . 
\end{displaymath}
This implies, in turn that 
\begin{displaymath}
 \{ g + f_{0}(\cdot,u(t;u_{0})), \ \eps\leq t \leq T, \ u_{0} \in B\}
 \quad \mbox{is bounded in $L^{q_{0}}(\Omega)$} 
\end{displaymath}
and again the smoothing effect of the semigroup implies that 
\begin{displaymath}
  \{ S(t) B,  \ \eps\leq t \leq T\}
 \quad \mbox{is bounded in $H^{2\gamma,q_{0}}(\Omega)
\cap H^{1,q_{0}}_0(\Omega))$}
\end{displaymath}
for any $\gamma <1$.

\bigskip

\noindent (ii) Notice that the proofs in \cite{Arrieta:2004p863},
\cite{cholewa09:_extrem} and \cite{Rodriguez-Bernal2002} are based on energy
estimates of the approximating solutions while the proof presented above is
based on the maximum principle, in the form of the comparison principle. In
particular, for the case of posing the problem in $L^1(\Omega)$, this avoid the
use of Kato's inequality providing a unified argument. The equivalence between
Kato's inequality and positive semigroups has been established in
\cite{arendt84:_katos}.

\bigskip

\noindent (iii) The standard theory for semilinear
reaction-diffusion equations requires $f$ to satisfy some growth restriction in
order to obtain a well-posed problem in $L^{q}(\Omega)$. Namely, the equation
(\ref{eq:pb}) is locally well posed provided $f$ satisfies
  \begin{equation}
    \label{eq:f:growth:restr}
    |f(x,t) - f(x,s)| \leq C(1 + |s|^{p-1} + |s|^{p-1})|t-s|, \quad t,s\in \R
  \end{equation}
  for all $x\in\Omega$,
  with
  \begin{displaymath}
  p \leq p_{c} = 1 + {2q \over N}
  \quad \left(\mathrm{i.e,\ for\ any}\quad q\geq q_C=\frac{N(p-1)}{2}\right).
\end{displaymath}

Notice that although the uniqueness in $L^{q}(\Omega)$, for $q>q_\mathrm{C}$,
when $f$ satisfies the growth restriction (\ref{eq:f:growth:restr}), follows
from with subcritical nonlinearities, the proof of Theorem \ref{thr:uniqueness}
does not use any growth restriction on the nonlinear term (other than the fact
of being almost-monotonic).

\bigskip

\noindent (iv) Theorems \ref{thr:truncated_problem_well_posed} and
\ref{thr:uniqueness} extend to problems in unbounded domains in a natural way
(see \cite{Arrieta:2004p863}).  Also, the same techniques can be applied to
obtain solutions in $\R^N$ in any $L^q_U(\Omega)$, locally uniform $L^q$ space,
see \cite{cholewa09:_extrem} for a proof based on energy estimates.  In the case
of initial data in $L^1_U(\Omega)$, $L$ in (\ref{eq:f:almostMonotonic}) was
required to be bounded. By the techniques presented in here, no additional
restriction is required on $L$ in order to obtain a solution.

\bigskip

\noindent (v) In \cite{marcus99:_initial}, positive solutions of equation $u_t -
\Delta u = -|u|^{p}$ with measures as initial data is considered. In
particular, for positive $L^1$ densities, the solution is unique. We have shown
that this uniqueness also holds for general $L^1$ initial data (with no
assumption on their sign).

\bigskip

\noindent (vi) An example of nonlinearity for which all the previous results
apply are the following:
\begin{displaymath}
  f_0(x,u) = \sum_{j=1}^k n_j(x) h_j(u) + f_1(x,u)
\end{displaymath}
with $h_j \in C^1(\R)$, $h_j(0) = h^\prime_j(0) = 0$, $j=1,\dots,k$,
and $f_1(x,s)$ is a H\"older continuous with respect to $x$ uniformly for
$s$ in bounded sets of $\R$, $\partial_s f_1(x,s)$ is bounded in
$x$ for $s$ in bounded sets of $\R$ and $f_1(x,0) = \partial_s
f_1(x,0) = 0$, $x\in \Omega$.
This includes in particular the following cases, taking $f_1\equiv 0$:
\begin{itemize}
\item \emph{Logistic equation}

\begin{displaymath}
  f_{0}(x,u)=  - n(x) |u|^{\rho -1} u
\end{displaymath}
with $n(x)$ a nonnegative $L^r(\Omega)$ function, not identically
zero, and $\rho >1$. In this case, $L_0(x,R) = \rho R^{\rho-1} n(x) $ and we
can always take $L\equiv 0$ in (\ref{eq:f:almostMonotonic}).

\item  Monotone polynomial nonlinearity
\begin{displaymath}
  f_{0}(x,u)=  \sum_{j=2}^{k} n_{j}(x) u^{j} 
\end{displaymath}
with $k$ odd and $n_{j}\in L^{d_j}(\Omega)$, $d_j>N/2$, $1\leq j\leq k$ and
$n_{k}(x) \leq a_{0} <0$ for all $x\in \Omega$. In this case, we can take
\begin{displaymath}
L_0(x,R) = \sum_{j=2}^{k} (j-1) R^{j-1}|n_{j}(x)|
\end{displaymath}
and
\begin{displaymath}
L(x) =k
\left[\max_{u\in \R}\sum_{j=2}^{k} n_{j}(x) u^{j-1} \right]^+,
\end{displaymath}
where $[g(x)]^+ = \max\{g(x),0\}$.  Notice that $L,L_0 \in L^{q_0}
(\Omega)$ with $q_0 = \min\{d_1,\dots,d_k\}$.

\item Polynomial nonlinearity with fractional powers 
\begin{displaymath}
  f_{0}(x,u)=  \sum_{j=1}^{k} n_{j}(x) |u|^{\rho_{j} -1} u  
\end{displaymath}
with $1< \rho_{j} < \rho_{k}$ and $ n_j \in L^{d_j}(\Omega)$, $d_j >N/2$, $1\leq
j \leq k$, and $n_k(x)\leq a_0<0$, $x\in \Omega$. We can take  $L$ and
$L_0$ analogous to the previous example. 

\end{itemize}

 \bibliographystyle{plain}

\end{document}